\documentclass{article}

\usepackage{epsfig}
\usepackage{amssymb, graphicx, amsmath, amsthm}
\usepackage{tikz}

\usepackage{amsrefs}
\usepackage{authblk}

\usepackage{graphicx}
\usepackage{hyperref}

\usetikzlibrary{patterns, shapes.geometric, matrix}

\theoremstyle{plain}
\newtheorem{thm}{Theorem}

\newtheorem{cor}{Corollary}

\newtheorem*{problem}{Problem}

\title{Envy-free and Approximate Envy-free Divisions of Necklaces and Grids of Beads}

\author[1]{Roberto Barrera}
\author[2]{Kathryn Nyman}
\author[3]{Amanda Ruiz}
\author[4]{Francis Edward Su}
\author[5]{Yan X Zhang}
\affil[1]{Department of Mathematics, Texas State University}
\affil[2]{Department of Mathematics, Willamette University}
\affil[3]{Department of Mathematics, University of San Diego}
\affil[4]{Department of Mathematics, Harvey Mudd College}
\affil[5]{Department of Mathematics, San Jose State University}

\date{\today}

\begin{document}
\maketitle

\begin{abstract}
We study discrete versions of the envy-free cake-cutting problem,
involving one-dimensional necklaces of beads and two-dimensional grids of beads.  
In both cases beads are indivisible, in contrast to cakes which are continuously divisible, but a general theme of our methods is an appeal to continuous cake-cutting ideas to achieve envy-free and approximate envy-free divisions of discrete objects.  (This problem is distinct from the ``necklace-splitting problem'' more commonly studied, which does not involve envy-freeness.) Our main result in two dimensions is an envy-free division of a grid of beads under certain conditions on the preferences, with near-vertical cuts.
\end{abstract}


\section{Introduction}
\label{sec:introduction}




The archetypal fair-division problem considers a division of cake among several players and seeks to find allocations that players consider ``fair''.  One such notion of fairness is that of \emph{envy-freeness}.  An allocation is called \emph{envy-free} if every player is happy with the piece they are assigned and would not prefer to trade with another player.  Under mild assumptions, the cake cutting problem has a solution (see e.g., \cite{su1999}): 

\begin{thm} [Envy-Free Cake-Cutting]
\label{thm:continuous existence}
For any set of $n$ players who prefer cake to no cake and have closed preference sets, there exists an envy-free allocation of cake using $(n-1)$ cuts. Furthermore, there exists a finite $\epsilon$-approximate algorithm.
\end{thm}

The existence of envy-free divisions has been known since Neyman \cite{neyman1946}, with recent attention paid to finding constructive proofs with potential for applications. The first constructive $n$-player envy-free solution is due to Brams and Taylor \cite{brams1995envy}.  It is a finite but unbounded procedure--- meaning that it will terminate but, depending on the preferences, the number of steps may be arbitrarily large!  Moreover, the cake could be divided into a huge number of pieces.  Aziz and Mackenzie \cite{mackenzie-aziz} recently developed a bounded $n$-person procedure, though even for small $n$ it can take an astronomical number of steps and cuts to resolve.

Other methods \cite{robertson1998cake,su1999} produce an \emph{approximate envy-free} division, i.e., a division in which each player feels their piece is within $\epsilon$ of being the best piece in their estimation.  An advantage to approximate procedures is that the number of steps and cuts is more manageable, and so has the possibility of being practical.  One such method, due to Simmons and described in \cite{su1999}, uses Sperner's lemma to accomplish a division by a minimal number of pieces.  Hence it requires only $(n-1)$ cuts.

In spirit of work such as \cite{marenco2011}, we consider a discrete analogue of the classical cake-cutting problem in which a number of indivisible beads along a line is to be allocated in an envy-free way.   We call this the problem of \emph{envy-free necklace-cutting}. Problems involving distributing indivisible goods has received considerable attention in the economics and fair division literature (see, e.g., \cite{bramsbook} and \cite{Lipton}), but with less emphasis on geometric constraints. Barbanel's work \cite{barbanel2005geometry} studies and organizes fair division under a geometric framework, but assumes divisible\footnote{Strictly speaking, \cite{barbanel2005geometry} cares about \emph{non-atomic} preferences, which is defined as when a player cannot put nontrivial value on a subset $A$ of a geometric space (analogous to a piece of cake) unless $A$ has a subset $B$ with strictly less but nonzero value to the player. However, this definition captures the intuition of divisible goods.} goods. Thus, we think our problem represents a potentially fruitful avenue of research as it combines having indivisible goods and having geometric constraints. Another representative of such problems is the \emph{necklace-splitting} problem, in which a string of $k$ types of beads is to be split among $n$ thieves so that each thief receives the same number of each type of bead.  (See e.g., \cite{alon}).  
This problem is different from our problem as it focuses on balancing the pieces for the $k$ types simultaneously and does not involve the thieves having different preferences for the beads. 

The necklace cutting problem is like the traditional cake-cutting problem, but the discreteness of the problem means an envy-free division may not necessarily exist.  The simplest example is when we have a single bead valued by every player.  If we give the bead to one player, all other players will be envious.


\begin{problem}
It is natural to consider the following as fundamental problems for studying envy-free necklace-cutting:
\begin{itemize}
\item What additional assumptions do we need to make to guarantee an envy-free division exists? 
\item What algorithm produces such divisions? 
\item When envy-free is impossible, what is the best we can do? For example, what is the smallest $\epsilon$ for which we can guarantee an $\epsilon$-envy-free division?
\end{itemize}
\end{problem}

Marenco and Tetzlaff \cite{marenco2011} initiated the study of envy-free necklace-cutting by considering the case where each bead is valuable to exactly one player.  (They don't use the language of necklace-cutting, but speak of atoms arranged on an interval.)  Appealing to the techniques in \cite{su1999} for traditional cake-cutting, they use a combinatorial result known as Sperner's lemma to find ``nearby'' divisions of the necklace, one for each player, in which each player prefers a different piece of the necklace in their division.  The location of cuts in these divisions differs by only one bead for each knife, and the idea is to use the location of the cuts to determine a final division that will be envy-free for all players.  The position of each knife in the final division is given by the position of that knife in the splitting associated to the player who values the contested bead.

We build on their work by examining the questions above for different classes of constraints, including more general preferences where several players may value each bead (Section \ref{sec:monolithic}), and a two-dimensional arrangement of the beads (Section~\ref{sec:2d}).   


\section{A cake-cutting model for envy-free necklace-cutting}

We wish to divide an open necklace of indivisible beads, but a recurring technique of this paper is to cut the cake analogue of a necklace (so that beads are potentially divided) and then slide the cuts so they do not divide beads.

We model this by considering the necklace as an interval $[0,\ell]$, in which the beads are now intervals between consecutive integers.  Viewed as a cake, cuts may be made at any point, but to be a necklace-cutting, cuts must lie at integers.  See Figure~\ref{fig:2 necklaces}.   

\begin{figure}[h]
\begin{center}
\begin{tabular}{cc}
\begin{tikzpicture}
\tikzstyle{A}=[circle,minimum size=0.3cm,draw=black,inner sep=0pt,outer sep=0pt];
\tikzstyle{edge} = [draw,line width=0pt,-];
\node[A](1)at(7,0.5){};
\node[A](2)at(8,0.5){};
\node[A](3)at(9,0.5){};
\node[A](4)at(10,0.5){};
\node[A](5)at(11,0.5){};
\path[edge] (1)--(2)-- (3)--(4)--(5);
\draw[very thick](7.5,-0.5)--(7.5,1.5);
\draw[very thick](8.5,-0.5)--(8.5,1.5);
\end{tikzpicture} & \begin{tikzpicture}
\draw[step=1cm,color=gray] (0,0) grid (5,1);
\draw[very thick](1,-0.5)--(1,1.5);
\draw[very thick](2,-0.5)--(2,1.5);
\end{tikzpicture} 
\end{tabular}
\caption{Two equivalent ways of picturing a discrete necklace cut into three strings of beads.
\label{fig:2 necklaces}}
\end{center}
\end{figure}
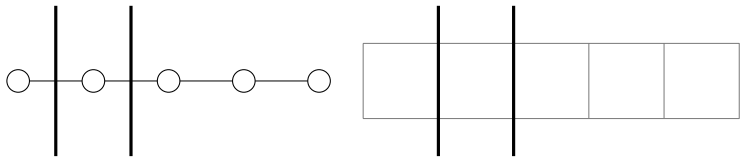

We assume each player $P$ has a nonnegative, additive, real-valued \emph{valuation} function $v_P$ defined on any possible string of beads, but because we will be cutting cake analogues of the necklace, we define $v_P$ on any fraction of a bead to be that same fraction of the valuation of a full bead.  In other words, we imagine the valuation of a bead to be uniformly distributed over the bead. 

This gives rise to a \emph{preference} relation between each two possible strings of beads, where
player $P$ \emph{prefers} string $A$ to $B$ if and only if $v_P(A) \geq v_P(B)$. We stress that a player may prefer several strings of beads if they value them equally. 
If $A$ is preferred to $B$ but not vice-versa, we say that the player \emph{strictly prefers} $A$ to $B$.

\section{Non-Monolithic Preferences}
\label{sec:monolithic}
\label{sec:general}

Marenco and Tetzlaff  \cite{marenco2011} showed that we can achieve an envy-free division of the necklace in the case where each bead is valued by only one player.  We say players have \emph{monolithic preferences} in this case, and we can label beads by the identities of the players who value them. See Figure \ref{beadscut} for an example. 

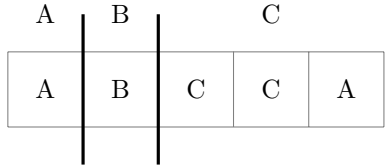
\begin{figure}[h]
\begin{center}
\begin{tikzpicture}
\draw[step=1cm,color=gray] (0,0) grid (5,1);
\draw(0.5, 0.5) node[anchor=center]{A};
\draw(1.5, 0.5) node[anchor=center]{B};
\draw(2.5, 0.5) node[anchor=center]{C};
\draw(3.5, 0.5) node[anchor=center]{C};
\draw(4.5, 0.5) node[anchor=center]{A};
\draw[very thick](1,-0.5)--(1,1.5);
\draw[very thick](2,-0.5)--(2,1.5);
\draw(0.5, 1.5) node[anchor=center]{A};
\draw(1.5, 1.5) node[anchor=center]{B};
\draw(3.5, 1.5) node[anchor=center]{C};
\end{tikzpicture}
\caption{A discrete necklace where players have monolithic preferences, cut by $2$ cuts into $3$ strings of beads. The labels in the necklace denote the players who value those beads; the labels above the necklace denotes the allocation. 
This particular allocation is envy-free if $A$ has the same valuation for both $A$-beads.
\label{beadscut}}
\end{center}
\end{figure}

\begin{thm}
\label{thm:main}(Marenco--Tetzlaff)
If $n$ players with monolithic preferences are to divide a necklace of beads, then there exists an envy-free division of the necklace using only $(n-1)$ cuts.
\end{thm}

We wish to study the case where each bead may be valued by more than one player.  
After all, the problem of fair division is most interesting when players bicker over items that are mutually desired.
Thus, unless otherwise stated, we assume valuation functions $v_P(b)$ can be positive for any player $P$ and for any bead $b$.  We've already seen an example in Section~\ref{sec:introduction} that we cannot guarantee an envy-free division in all cases.  A less trivial example is given in Figure~\ref{fig:multiple-prefs}.

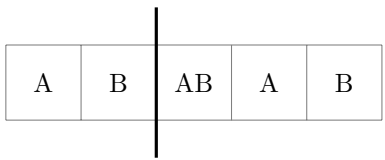
\begin{figure}[h]
\begin{center}
\begin{tikzpicture}
\draw[step=1cm,color=gray] (0,0) grid (5,1);
\draw(0.5, 0.5) node[anchor=center]{A};
\draw(1.5, 0.5) node[anchor=center]{B};
\draw(2.5, 0.5) node[anchor=center]{AB};
\draw(3.5, 0.5) node[anchor=center]{A};
\draw(4.5, 0.5) node[anchor=center]{B};
\draw[very thick] (2,-0.5) -- (2,1.5);
\end{tikzpicture}
\caption{If each player values each bead with their label equally, then there is no envy-free division between players $A$ and $B$; they will fight over the middle bead. 
\label{fig:multiple-prefs}}
\end{center}
\end{figure}

However, just because we cannot guarantee an envy-free division does not mean we cannot come close. As in \cite{su1999}, call a division \emph{$\epsilon$-envy-free} if for each player, the other strings are worth no more than $\epsilon$ plus the value of the string assigned to them. Note that a division is $0$-envy-free if and only if it is envy-free. We achieve the following result:

\begin{thm}
\label{thm:general} 
For a necklace of beads, where the value of each bead is at most $s$ to every player, there always exists an $\epsilon$-envy-free division of the necklace among $n$ players using only $(n-1)$ cuts with $\epsilon < 2s$. In particular, there is always a $(2s)$-envy-free division.
\end{thm}

\begin{proof}
Represent our necklace of $\ell$ beads as the interval $[0,\ell]$ such that beads are the intervals between consecutive integers and players' valuations of each bead are uniformly distributed over the bead's interval. If we allow the cuts to be made at non-integer points, Theorem~\ref{thm:continuous existence} guarantees that there exists an envy-free division, $D$, of the interval. We begin with the division $D$ and shift any cuts that divide a bead by rounding each such cut to the nearest integral point.  For example, if a cut is at $x = 1.2$, shift it so it is at $x = 1.0$. If a cut is at a half-integer, then we always round to the right. 

Suppose player $P$ were assigned a string of beads with value $v$ in $D$. Then this rounding process decreases $P$'s valuation by at most $s/2$ at the left end and strictly less than $s/2$ at the right end, so $P$'s valuation of her piece decreases by strictly less than $s$. Furthermore, from the perspective of $P$, the value of any other piece in $D$ increases by at most $s$ by the same logic.  Thus, $P$'s envy in the final division, where all the cuts are now at integer points, is strictly less than $2s$. 
\end{proof}

If we specialize to the situation where the valuations are integral, the strictness of the inequality above can be quite useful. One such situation is when the only opinion a player has towards a bead is that they either desire the bead or are indifferent to the bead.  We model this situation by assigning each bead's worth to a player as either exactly $1$ or $0$ respectively. In this case, we  obtain:

\begin{cor}
For $n$ players dividing a necklace of beads with valuations taking values either 0 or 1, there always exists a $1$-envy-free division using only $(n-1)$ cuts.
\end{cor}
\begin{proof}
Theorem~\ref{thm:general} shows that we can get the maximum envy to be strictly less than $2$. However, since the valuations are sums of $1$'s and $0$'s, the envy must be an integer and so the envy is bounded above by $1$.
\end{proof}

\section{Two Dimensional Divisions: a Grid of Beads}
\label{sec:2d}

In this section, we consider a $2$-dimensional grid of indivisible beads that we wish to divide in an envy-free way. We can imagine that there are strings in the $4$ cardinal directions connecting the beads; we can also just imagine that the beads are on a valueless ``quilt'' which is to be cut.  This is a generalization of cake-cutting problem in two-dimensions, which has been studied (e.g., \cite{segal-halevi}), but for divisible goods.  An important consideration in two-dimensional division is the geometry of the pieces. Our main theorem is Theorem \ref{thm:2d}, which uses a ``sliding'' method remniscent of network flow problems from computer science.

Segal-Halevi (personal communication, Nov.\,2015) 
has pointed out that an envy-free division of a grid of beads can be accomplished by the Marenco-Tetzlaff result by cutting the grid into a one-dimensional necklace along a snake-like path that winds back and forth across the rows of the grid, from top to bottom. (In fact, any Hamiltonian path along the grid graph would do.) This would achieve an an envy-free division with pieces that are connected in the original grid.  He also notes that in the case of preferences that are not necessarily monolithic, our Theorem \ref{thm:general} would yield an $(2s)$-envy-free division where $s$ is the maximum valuation of any bead by any player.

However, we desire our pieces to align with the geometry of the grid in near-vertical cuts, which could be useful if, for instance, the rectangular grid represented a grid of resources to be split up by players, and for some geometric reason, we need the pieces to ``connect'' vertically to resources at the top and the bottom of the grid.  Perhaps the top of the grid is access to a waterway useful for trade, and the bottom of the grid is access to a canyon of raw materials.  This situation may be contrived, but we believe the geometric constraint is interesting and the methods generated by this paper useful for further analysis.

If we demand actual vertical cuts, Theorem~\ref{thm:general} yields an immediate corollary:



\begin{cor}
\label{cor:easycor}
Consider a $2$-dimensional grid of beads. Suppose $s$ is the highest total valuation of the beads in a column by any player. Then there exists an $\epsilon$-envy-free division among $n$ players using $(n-1)$ vertical cuts, where $\epsilon < 2s$.
\end{cor}

\begin{proof}
We model our grid of beads as the region $[0,k] \times [0,\ell]$ such that each square is a bead and the players' valuations are spread uniformly across each bead. We can now treat our grid as a $1 \times \ell$ necklace by summing the valuations of the $i$-th column to obtain the valuation of the $i$-th bead. Applying Theorem~\ref{thm:general} gives the result. 
\end{proof}

However, it seems that we can do significantly better if we relax our requirement so that cuts are allowed to be ``near-vertical''.  For what follows, define a \emph{near-vertical cut} of a grid to be a path of edges, all lying within one vertical column of the grid, that separate the squares into two sets (as a Jordan curve).  We consider two cuts to be \emph{non-intersecting} if the cuts do not intersect each other transversely, that is, there do not exist two points of a cut such that they are in the interiors of the two different sets created by the other cut.

\begin{thm}
\label{thm:2d}
Suppose $n$ players have monolithic preferences over a 2-dimensional grid of beads with valuations taking values either 0 or 1. Then there exists an envy-free division of the grid using $(n-1)$ non-intersecting near-vertical cuts.
\end{thm}

\begin{figure}[h]
\begin{center}
\begin{tabular}{cc}
\begin{tikzpicture}
\draw[step=1cm,color=gray] (0,0) grid (3,5);
\fill[pattern=north east lines] (1,0) rectangle (2,1);
\fill[pattern=north east lines] (1,2) rectangle (2,3);

\fill[pattern=north east lines] (0,3) rectangle (1,4);
\fill[pattern=north east lines] (1,4) rectangle (2,5);
\fill[pattern=dots] (1,3) rectangle (2,4);
\fill[pattern=dots] (2,0) rectangle (3,5);
\fill[pattern=dots] (1,1) rectangle (2,2);

\draw[very thick] (1, -0.5) -- (1,1)  -- (1, 3) -- (0, 3) -- (0, 4) -- (1, 4) -- (1, 5.5);
\draw[very thick] (2, -0.5) -- (2,1) -- (1,1) -- (1,2) -- (2,2) -- (2, 3) -- (1, 3) -- (1, 4) -- (2, 4) -- (2, 5.5);
\end{tikzpicture} & \begin{tikzpicture}
\draw[step=1cm,color=gray] (0,0) grid (3,5);
\tikzstyle{A}=[circle,minimum size=0.3cm,draw=black,inner sep=0pt,outer sep=0pt];

\foreach \y in {+0.5,+1.5,2.5,3.5, 4.5}
  \foreach \x in {0.5,1.5,2.5}
    \node [A] at (\x, \y) {};

\draw[very thick] (0.8, -0.5) -- (0.8,3)  -- (0.2,3) -- (0.2,4) -- (0.8, 4) -- (0.8, 5.5);
\draw[very thick] (1.8, -0.5) -- (1.8, 1) -- (1.2, 1) -- (1.2,2) -- (1.8,2) -- (1.8,3) -- (1.2, 3) -- (1.2, 4) -- (2, 4) -- (2, 5.5);
\end{tikzpicture}

\end{tabular}
\caption{The grid is cut via $2$ cuts into $3$ pieces, each (barely) path-connected if we allow movement along the cuts themselves. The right is a visualization of this division interpreted as beads on a valueless ``quilt.'' 
\label{fig:2d-connectivity}}
\end{center}
\end{figure}
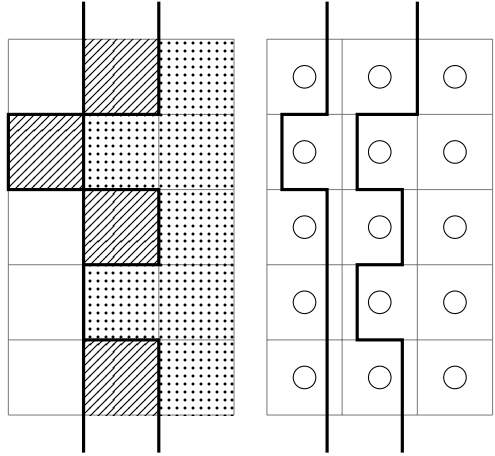

\begin{proof}
Again, we model our grid of beads as the region $[0,k] \times [0,\ell]$. By Theorem~\ref{thm:continuous existence}, there exists an envy-free division of the grid using $(n-1)$ parallel vertical cuts. We think of each of these vertical cuts as a union of $k$ vertical edges. We call such a vertical edge \emph{integral} if its $x$-coordinate is integral and \emph{non-integral} otherwise. Our strategy is to slide non-integral vertical edges left or right so they become integral while keeping them connected via horizontal edges along the integral $y$-coordinates. We do this while holding the relative left-to-right order of the $(n-1)$ vertical edges on each horizontal strip constant, thus generating the desired cuts for the discrete grid of squares. 

We say that a vertical edge \emph{borders} players $P$ and $Q$ if the two adjacent pieces allocated by the original envy-free allocation belong to $P$ and $Q$. First, consider any bead $S$ desired by player $P$ such that at least one vertical edge going through $S$ borders $P$. In this case, we can allocate the bead completely to player $P$ (by moving all edges in $S$ left of $P$'s piece to $S$'s left boundary and all edges in $S$ right of $P$'s piece to $S$'s right boundary). Due to monolithic preferences, this gives player $P$ a strictly more valuable piece and does not affect the preferences and envy of other players, as only $P$ cares about the bead. Doing this for all beads ensures two conditions now hold for our (still envy-free) division $D$: 
\begin{itemize}
\item for every player, $P$, $D$ assigns an integral amount of beads that player $P$ values to player $P$ (if not, then there must be some non-integral vertical edge bordering $P$ somewhere);
\item if a bead desired by $P$ contains a non-integral vertical edge bordering $P_1$ and $P_2$, then neither $P_1$ nor $P_2$ can be $P$.
\end{itemize}

Our strategy to move the remaining non-integral edges is as follows. Suppose we have at least one non-integral vertical edge somewhere in our envy-free division. We describe a sliding process such that:
\begin{itemize}
\item we stay envy-free at all times,
\item any time a non-integral edge slides into another, we consider the two edges to have merged into a single edge (and sliding the resulting edge corresponds to the underlying edges moving together as a group), and
\item any time a non-integral edge becomes integral, we no longer slide it.
\end{itemize}
When a non-integral edge slides into another edge or becomes integral, the number of non-integral edges decreases by one; we can then repeat this process until all non-integral edges become integral. Thus, it suffices to show that we can always perform this sliding process without losing the envy-free property.

Consider any bead $S$ with non-integral vertical edges going through it, desired by player $P$. Say that $S$ is \emph{contested} by $Q$ if $Q$'s piece contains part of $S$. Let the set of players contesting $S$ contain $P_1$ and $P_2$ (recall that neither can be $P$). We say we \emph{donate from $P_1$ to $P_2$ through $S$} if we slide the non-integral vertical edges through $S$ in a (unique) way that the $P_1$ piece in $S$ decreases at a constant speed, the $P_2$ piece in $S$ increases at the same constant speed, and the other pieces in $S$ stay at constant size. Donating from $P_1$ to $P_2$ through $S$ keeping the division envy-free results in one of the following:
\begin{itemize}
\item the $P_1$ part of $S$ becomes empty, which causes the number of non-integral edges to decrease, or
\item $P$ values $P_2$'s piece exactly equal to her own, in which case any further donation to $P_2$ from $P_1$ would make $P$ envious of $P_2$.
\end{itemize}

Suppose there exists at least one remaining bead desired by player $P$ with at least one non-integral edge going through it.  We define an auxiliary graph $G$ with vertices indexed by players who are not $P$. For each bead $S$ desired by player $P$ with at least one non-integral edge, draw an edge in $G$ labeled by $S$ with endpoints $P_1$ and $P_2$ for every pair such that a non-integral edge in $S$ borders $P_1$ and $P_2$ 
(we allow multiple edges in $G$). We have two cases:

\begin{itemize}
\item If we have a cycle (with no repeated vertices) in $G$ of the form 
\[
P_1 \rightarrow P_2 \rightarrow \cdots \rightarrow P_k \rightarrow P_1,
\]
then we can simultaneously donate from $P_1$ to $P_2$ (through the bead corresponding to the edge between $P_1$ and $P_2$ in $G$), $P_2$ to $P_3$, etc. through $P_k$ to $P_1$. Since $P$'s valuation of all the $P_i$'s pieces are constant (each piece is donating and being donated to at the same rate), at some point at least one of the parts belonging to some $P_i$ in one of these beads becomes $0$, corresponding to a decrease in the number of non-integral vertical edges. 
\item If we do not have a cycle, then some vertex must have degree $1$. This means there is some bead $S$ desired by $P$ and contested by $Q \neq P$ where $Q$ does not have a fractional piece of a bead desired by $P$ anywhere else. This means $P$'s valuation of $Q$'s piece has fractional part exactly equal to $Q$'s amount in $S$. Because $P$'s evaluation of $P$'s own piece is currently integral, we can give the entire bead $S$ to $Q$ without fear that $P$ will become envious of $Q$. 
\end{itemize}
For an example of this process, see Figure ~\ref{fig:sliding}. 

\begin{figure}[h]
\begin{center}
\begin{tabular}{cc}
\begin{tikzpicture}
\draw[step=1cm,color=gray] (0,0) grid (4,2);
\node at (0.5,0.5) {C};
\node at (0.5,1.5) {C};
\node at (1.5,0.5) {A};
\node at (1.5,1.5) {A};
\node at (2.5,0.5) {B};
\node at (2.5,1.5) {B};
\node at (3.5,0.5) {A};
\node at (3.5,1.5) {B};

\node at (0.75,2.5) {C};
\node at (2.25,2.5) {B};
\node at (3.5,2.5) {A};

\node [draw=none] at (1.5, 1.5) (hu) {};
\node [draw=none] at (1.5, 0.5) (hl) {};
\node [draw=none] at (0.5, 1.5) (hu2) {};
\node [draw=none] at (2.5, 0.5) (hl2) {};
\draw [->] (hu) -- (hu2);
\draw [->] (hl) -- (hl2);

\draw[very thick] (1.5, -0.5) -- (1.5, 2.5);
\draw[very thick] (3, -0.5) -- (3, 2.5);

\end{tikzpicture} & \begin{tikzpicture}
\draw[step=1cm,color=gray] (0,0) grid (4,2);
\node at (0.5,0.5) {C};
\node at (0.5,1.5) {C};
\node at (1.5,0.5) {A};
\node at (1.5,1.5) {A};
\node at (2.5,0.5) {B};
\node at (2.5,1.5) {B};
\node at (3.5,0.5) {A};
\node at (3.5,1.5) {B};

\node at (0.5,2.5) {C};
\node at (2,2.5) {B};
\node at (3.5,2.5) {A};

\draw[very thick] (2, -0.5) -- (2,1) -- (1,1) -- (1, 2.5);
\draw[very thick] (3, -0.5) -- (3, 2.5);

\end{tikzpicture}

\end{tabular}
\caption{Left: an envy-free cut where $A$ values all $3$ pieces equally. We wish to move the non-integral vertical edges, but doing either in isolation would cause $A$ envy. Thus, we slide the top edge left at the same rate that the bottom edge slides right, keeping $A$ envy-free, until they slide to integral points.
\label{fig:sliding}}
\end{center}
\end{figure}
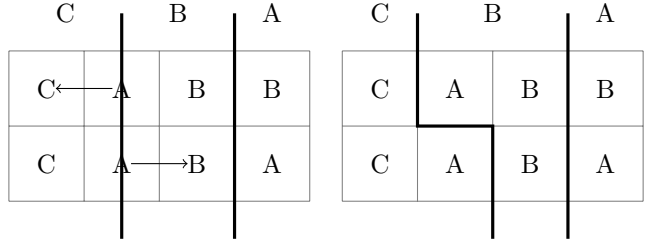

In both cases, we strictly decrease the number of non-integral vertical edges. Thus, we are able to perturb the vertical edges until all edges are integral so that the cut is near-vertical, in which case we have an envy-free allocation of the grid of beads.
\end{proof}

From the proof, we observe that a vertical edge through a bead $S$ in the original division can move to the left or right boundary of $S$. This means in the worst case, we can get some disjoint-looking pieces, with each vertical cut having a possible horizontal deviation of distance at most $1$; for an example, see Figure~\ref{fig:2d-connectivity}. However, we maintained the vertical nature of the pieces, as desired.

This proof only uses the 0-1 constraint in the second bullet point above when resolving the degree 1 nodes.  It would be interesting to see if this proof could be extended to more general preferences.

\section{Acknowledgments}

This work began at the 2014 AMS Mathematics Research Community - Algebraic and Geometric Methods
in Applied Discrete Mathematics, which was supported by NSF DMS-1321794. It was finished while two of the authors were in residence at at the Mathematical Sciences Research Institute in Berkeley, California, during the Fall 2017 semester where they were supported by the National Science Foundation under Grant No. DMS-1440140.
Su was supported partially by NSF Grant DMS-1002938. We thank Boris Alexeev for valuable conversation. We also thank Erel Segal-Halevi for his valuable comments in Section \ref{sec:2d}. 

\bibliographystyle{abbrv}
\bibliography{extension_latest}

\end{document}